\documentclass[a4paper,12pt]{article}

\usepackage{amscd,amssymb,amsmath,amsfonts,amsthm, mathrsfs, latexsym}
\usepackage{graphicx}
\usepackage{verbatim,enumerate,paralist}
\usepackage{textcomp}
\usepackage[arrow, curve, frame, matrix, graph, tips]{xy}

\usepackage[colorlinks]{hyperref}
\hypersetup{%
  urlcolor=blue,linkcolor=blue,citecolor=blue
}
\usepackage{pdfpages}

\usepackage[T1]{fontenc}

\newcounter{Definitioncount}
\newtheorem{theorem}{Theorem}

\newtheorem{proposition}[theorem]{Proposition}
\newtheorem{corollary}[theorem]{Corollary}

\theoremstyle{definition}
\newtheorem*{Definition}{Definition}

\newtheoremstyle{fact}{\bigskipamount}{\medskipamount}{\upshape}{}{\itshape}{. }{ }{Fact}
\theoremstyle{fact}

\newtheoremstyle{genquest}{\bigskipamount}{\medskipamount}{\upshape}{}{\itshape}{. }{ }{General Question}
\theoremstyle{genquest}

\newtheoremstyle{step}{2\bigskipamount}{\medskipamount}{\upshape}{}{\itshape}{. }{ }{\underline{Step~\thestep}}
\theoremstyle{step}

\renewcommand{\thestep}{\arabic{step}}

\newcommand{\ra}{\rightarrow}
\newcommand{\lra}{\longrightarrow}
\newcommand{\lla}{\longleftarrow}
\newcommand{\Lra}{\Longrightarrow}

\newcommand{\ldual}[1]{\mathord{{\let\nolimits\relax\sideset{^\wedge}{}{#1}}}}
\newcommand{\laction}[2]{\mathord{{\let\nolimits\relax\sideset{^{#1}}{}{#2}}}}
\newcommand{\conj}[2]{\mathord{{\let\nolimits\relax\sideset{^{#1}}{}{#2}}}}

\def\CA{{\mathscr A}}
\def\CB{{\mathscr B}}
\def\CC{{\mathscr C}}

\def\CK{{\mathscr K}}

\def\CM{{\mathscr M}}

\def\CS{{\mathscr S}}

\def\CU{{\mathscr U}}
\def\CV{{\mathscr V}}

\def\CX{{\mathscr X}}

\begin{document}
\author{Ross Street \footnote{The author gratefully acknowledges the support of Australian Research Council Discovery Grant DP130101969. This material was presented to the Australian Category Seminar at Macquarie University on 14 and 21 August 2013.} \\ 
\small{Macquarie University, NSW 2109 Australia} \\
\small{<ross.street@mq.edu.au>}
}
\title{Pointwise extensions and sketches in bicategories}
\date{\small{\today}}
\maketitle

\noindent {\small{\emph{2010 Mathematics Subject Classification:} 18D10; 18D05}
\\
{\small{\emph{Key words and phrases:} bicategory; pointwise Kan extension; right lifting; cofibration; collage; lax limit.}}

\begin{abstract}
\noindent We make a few remarks concerning pointwise extensions in a bicategory which 
include the case of bicategories of enriched categories. We show that extensions, pointwise or not, 
can be replaced by extensions along very special fully faithful maps.
This leads us to suggest a concept of limit sketch internal to the bicategory.     
   
\end{abstract}

\tableofcontents

\section{Extensions and liftings in a bicategory}\label{Elb}

We begin by recalling the concepts of extension and lifting as used, for example, in \cite{FYL} and \cite{YS}.
Let $\CM$ be a bicategory. 
Consider a triangle
\begin{equation}\label{maintriangle}
\begin{aligned}
\xymatrix{
A \ar[rd]_{f}^(0.5){\phantom{a}}="1" \ar[rr]^{j}  && B \ar[ld]^{r}_(0.5){\phantom{a}}="2" \ar@{<=}"1";"2"^-{\rho}
\\
& X 
}
\end{aligned}
\end{equation}
in $\CM$. We say the 2-morphism $\rho$ in \eqref{maintriangle} exhibits $r$ 
as a {\em right extension} of $f$ along $j$ when pasting
the triangle to 2-morphisms $\theta : g \Lra r$ provides a bijection
between such 2-morphisms $\theta$ and 2-morphisms $\phi :  gj \Lra f$.
Then $r$ is unique up to a unique isomorphism.   
We write $\mathrm{hom}_A (j,f)$ for a right extension of $f : A\lra X$ along $j : A\lra B$.
To say every $f$ has a right extension along $j$ is to say the functor
$$\CM (j,1) : \CM (B,X)\lra \CM (A,X) \ ,$$
defined by precomposition with $j$, has a right adjoint $\mathrm{ran}_j$;
so $\mathrm{ran}_jf = \mathrm{hom}_A (j,f)$. If also $i:B\lra C$ then 
\begin{equation}\label{comphome}
\mathrm{hom}_A (ij,f) \cong \mathrm{hom}_B (i , \mathrm{hom}_A (j,f)) \ .
\end{equation}

{\em Left extensions} in $\CM$ are right extensions in $\CM^{\mathrm{co}}$:
we write $\mathrm{lan}_jf$ for a left extension of $f$ along $j$. 

{\em Right liftings} in $\CM$ are right extensions in $\CM^{\mathrm{op}}$.
We write $\mathrm{hom}^B (m,n)$ for a right lifting of $n:X\lra B$ through $m : A\lra B$.
If also $\ell : C \lra A$ then
\begin{equation}\label{comphoml}
\mathrm{hom}^B (m\ell,n) \cong \mathrm{hom}^A (\ell , \mathrm{hom}^B (m,n)) \ .
\end{equation}

A {\em left preadjoint} for $n : X\lra B$ is a right extension $\mathrm{hom}_X (n,1_X)$ of 
the identity 1-morphism $1_X$ of $X$ along $n$. Left adjoints are particular left preadjoints:
when $j\dashv j^*$, we have
$$j \cong \mathrm{hom}_A (j^*,1_A) \ .$$

Left adjoints are sometimes called {\em maps}.
A map $j:A\lra B$ with right adjoint $j^*:B\lra A$ is {\em fully faithful} when the unit 
$\eta_j : 1_A\Lra j^*j$ is invertible. 

\section{Pointwise extensions}\label{Ple}

Suppose $\CK$ is a class of maps of $\CM$ which is closed under composition and includes the identities.
We regard $\CK$ as a locally full sub-bicategory of $\CM$. 
When regarding a morphism $f$ of $\CK$ as in $\CM$, we write it as $f_*$. 
This situation, including the next definition, was studied by Richard Wood \cite{WoodPAI, WoodPAII}.

\begin{Definition}
If \eqref{maintriangle} is in $\CK$, we say it is a {\em pointwise right extension}
in $\CK$ when it is a right extension in $\CM$. So 
\begin{equation}
r \cong \mathrm{hom}_A (j_*,f) \ .
\end{equation}   
\end{Definition}

With this definition, the first proposition is obvious.

\begin{proposition}
If \eqref{maintriangle} is a pointwise right extension in $\CK$ then it is a right extension in $\CK$.  
\end{proposition}

\begin{proposition}
If \eqref{maintriangle} is a pointwise right extension in $\CK$ and $j$ is 
fully faithful then $\rho$ is invertible.  
\end{proposition}
\begin{proof}
The right extension of $r$ through $j^*$ is $rj$. 
Using \eqref{comphome}, we see that $rj$ is the right extension of $f$ through $j^*j_*\cong 1_A$.
The right extension through an identity is the same 1-morphism.
So $rj\cong f$. One of the adjunction identities for $j_*\dashv j^*$ shows that this isomorphism is $\rho$.
\end{proof}

A 1-morphism $j:B\lra C$ is called a {\em cofibration} when each $\CM(j,X): \CM(C,X)\lra \CM(B,X)$ is a fibration
in the bicategorical sense of \cite{FiB}.
A 1-morphism $i:A\lra C$ is called a {\em coopfibration} when each $\CM(i,X): \CM(C,X)\lra \CM(A,X)$ is 
an opfibration in the bicategorical sense of \cite{FiB}.

If $u:A\lra B$ is a 1-morphism of $\CK$, an object $X$ is {\em $u$-complete} when every $f: A \lra X$
has a pointwise right extension along $u$. In that case, if $\CM$ admits right extensions, the adjunction
$$\CM(u_*,1)\dashv \mathrm{hom}_A(u_*,-) : \CM(A,X) \lra \CM(B,X)$$ restricts to an adjunction
$$\CK(u,1)\dashv \mathrm{ran}_u : \CK(A,X) \lra \CK(B,X) \ .$$

A {\em collage} or {\em lax colimit} \cite{CCEC} of $j:A\lra B$ in $\CM$ is a universal 2-morphism
\begin{equation}\label{coll}
\begin{aligned}
\xymatrix{
A \ar[rd]_{i_1}^(0.5){\phantom{a}}="1" \ar[rr]^{j}  && B \ar[ld]^{i_0}_(0.5){\phantom{a}}="2" \ar@{<=}"1";"2"^-{\nu}
\\
& \mathrm{C}(j) & . }
\end{aligned}
\end{equation}
That is, pasting with $\nu$ determines an equivalence of categories
\begin{equation}\label{univprop}
\CM (\mathrm{C}(j),X) \simeq \CM (j,X)/\CM (A,X) \ ,
\end{equation}
where the slice category on the right-hand side is the comma category \cite{CWM}
of the functor $\CM (j,X)$ and the identity functor of $\CM (A,X)$. 

\begin{proposition} In the collage \eqref{coll}, the 1-morphism $i_0$ is a map with invertible unit $1_B \cong i_0^*i_0$, and with $i_0^* i_1\cong j$ compatibly with $t\nu$. 
If $\CM$ has local initial objects, preserved by composition with each given 1-morphism, and if $i_1$ is a map, then the unit $1_A \cong i_1^*i_1$ is invertible. 
Furthermore, $i_0$ is a cofibration and $i_1$ is a coopfibration.
\end{proposition}

Examples in \cite{CCEC} suggest consideration of the following properties of a collage. 

\begin{Definition} We say that a collage \eqref{coll} is {\em $\CK$-compatible} when 
\begin{enumerate}
\item $i_0$ and $i_1$ are in $\CK$,
\item for any $h:\mathrm{C}(j) \lra X$, if $hi_0$ and $hi_1$ are in $\CK$, so is $h$, and 
\item the mate triangle 
\begin{equation}\label{laxlim}
\begin{aligned}
\xymatrix{
& \mathrm{C}(j) \ar[ld]_{i_1^*}^(0.5){\phantom{a}}="1"   \ar[rd]^{i_0^*}_(0.5){\phantom{a}}="2" \ar@{=>}"1";"2"^-{\hat{\nu}}
\\
A \ar[rr]_-{j} && B 
}
\end{aligned}
\end{equation}
exhibits $\mathrm{C}(j)$ as a lax limit of $j$ in $\CM$.
\end{enumerate}
\end{Definition} 

From the universal property \eqref{univprop}, the triangle \eqref{maintriangle} induces a 1-morphism
$\langle r, \rho \rangle : \mathrm{C}(j) \lra X$ and invertible 2-cells $f\cong \langle r, \rho \rangle i_1$ and $r \cong \langle  r, \rho \rangle i_0$
which conjugate $\langle  r, \rho \rangle \nu$ to give $\rho$.  

Now consider the triangle
\begin{equation}\label{cofibtriangle}
\begin{aligned}
\xymatrix{
A \ar[rd]_{f}^(0.5){\phantom{a}}="1" \ar[rr]^{i_1}  && \mathrm{C}(j) \ar[ld]^{\langle r, \rho \rangle}_(0.5){\phantom{a}}="2" \ar@{<=}"1";"2"^-{\cong}
\\
& X 
}
\end{aligned}
\end{equation}

\begin{proposition}\label{equivformext}  
In $\CM$, triangle \eqref{cofibtriangle} exhibits $\langle r, \rho \rangle$ as a right extension of $f$ along $i_1$ if and only if triangle \eqref{maintriangle} exhibits $r$ as a right extension of $f$ along $j$.
\end{proposition} 
\begin{proof} 
There is a right extension
\begin{equation}\label{i0triangle}
\begin{aligned}
\xymatrix{
\mathrm{C}(j) \ar[rd]_{\langle r, \rho \rangle}^(0.5){\phantom{a}}="1" \ar[rr]^{i_0^*}  && B \ar[ld]^{\langle r, \rho \rangle i_0}_(0.5){\phantom{a}}="2" \ar@{<=}"1";"2"^-{\langle r, \rho \rangle \varepsilon}
\\
& X 
}
\end{aligned}
\end{equation}
coming from the adjunction $i_0 \dashv i_0^*$.
If \eqref{cofibtriangle} is a right extension, we just paste the right extension \eqref{i0triangle} to the right side of it and use the isomorphisms $i_0^*i_1 \cong j$ and $\langle r,\rho \rangle i_0 \cong r$ 
to see that \eqref{maintriangle} is a right extension.  

Conversely, take $h:\mathrm{C}(j) \lra X$ and $\theta : hi_1\Lra f$. By the universal property \eqref{univprop} of $\mathrm{C}(j)$,
in order to give a 2-morphism $h \Lra \langle r, \rho \rangle $, we precisely require two 2-morphisms
$\theta : hi_1\Lra f$ and $\phi : hi_0 \Lra r$ such that $\theta (h\nu) = \rho (\phi j)$. However, each $\theta$
determines such a unique $\phi$ by the right extension property of \eqref{maintriangle}. So \eqref{cofibtriangle}
is a right extension.   
\end{proof}

Now by applying Proposition~\ref{equivformext} to triangle \eqref{maintriangle} thought of as in $\CM^{\mathrm{coop}}$, we obtain a pointwise version.
 
\begin{corollary}
Suppose the triangle \eqref{maintriangle} is in $\CK$. If \eqref{cofibtriangle} exhibits $\langle r, \rho \rangle$ as a pointwise right extension of $f$ along $i_1$ then \eqref{maintriangle} exhibits $r$ as a pointwise right extension of $f$ along $j$.
If the collage \eqref{coll} is $\CK$-compatible then the converse also holds. 
\end{corollary}

\begin{Definition}
A square 
\begin{equation}\label{BC}
\begin{aligned}
\xymatrix{
S \ar[d]_{q}^(0.5){\phantom{aaa}}="1" \ar[rr]^{p}  && U \ar[d]^{b}_(0.5){\phantom{aaa}}="2" \ar@{<=}"1";"2"^-{\lambda}
\\
A \ar[rr]_-{j} && B 
}
\end{aligned}
\end{equation}
in $\CK$ is called {\em BC} or, by Ren\'e Guitart \cite{Guitart1980}, {\em exact} when its mate square 
\begin{equation}\label{BCmate}
\begin{aligned}
\xymatrix{
S \ar[rr]^{p}   && U 
\\
A \ar[rr]_-{j}  \ar[u]^{q^*}_(0.5){\phantom{aaa}}="2" && B \ar[u]_{b^*}^(0.5){\phantom{aaa}}="1" \ar@{<=}"1";"2"_-{\overline{\lambda}}  
}
\end{aligned}
\end{equation}
in $\CM$ has $\overline{\lambda}$ invertible.
\end{Definition}

\begin{proposition}
If the triangle \eqref{maintriangle} is a pointwise right extension and the square \eqref{BC} is BC then the
2-morphism obtained by pasting the square on top of the triangle exhibits $rb$ as a right extension
of $fq$ along $p$. 
\end{proposition}
\begin{proof}
Take $h:U\lra X$. Then $\theta : hp \Lra fq$ are in bijection with $pq^*\Lra h^*f$ which, by the BC property,
are in bijection with $b^*j\Lra h^*f$, and so with $hb^*j\Lra f$. By the pointwise right extension property of \eqref{maintriangle}, these are in bijection with $hb^* \Lra r$, and so, by mates, with $\phi : h \Lra rb$. 
Then $\lambda$, $\rho$ and $\phi$ paste to give back $\theta$, as required.  
\end{proof}

\section{The enriched category example}

Let $\CV$ be a base monoidal category as used in \cite{KellyBook}. 
For two-sided modules (also called bimodules, profunctors and distributors) refer to 
\cite{LawMetric}, \cite{CCEC} and \cite{MbHa}, for example.
  
Let $\CV\text{-}\mathrm{Mod}$ be the bicategory whose objects are $\CV$-categories $\CA$ 
(which are small relative to $\CV$),
whose 1-morphisms are two-sided modules $M : \CA \lra \CB$  
and whose 2-morphisms are module morphisms.
We have an equivalence of categories 
\begin{equation}\label{modhom}
\CV\text{-}\mathrm{Mod}(\CA,\CB)\simeq [\CB^{\mathrm{op}}\otimes \CA,\CV] \ .
\end{equation}
The composite $N\circ M$ of modules $M : \CA \lra \CB$ and $N : \CB \lra \CC$
is defined by the coend formula:
\begin{eqnarray}
(N\circ M)(C,A)=\int^{B}{M(B,A)\otimes N(C,B)} \ .
\end{eqnarray}
Notice that the hom categories \eqref{modhom} are cocomplete and composition is 
separately colimit preserving (local cocompleteness). 
Indeed, right liftings and right extensions all exist in
$\CV\text{-}\mathrm{Mod}$. In particular, we have
\begin{equation}\label{maintrianglemod}
\begin{aligned}
\xymatrix{
\CA \ar[rd]_{N}^(0.5){\phantom{a}}="1" \ar[rr]^{M}  && \CB \ar[ld]^{hom_{\CA}(M,N)}_(0.5){\phantom{a}}="2" \ar@{<=}"1";"2"^-{\mathrm{ev}}
\\
& \CX 
}
\end{aligned}
\end{equation}
where
\begin{eqnarray}
\begin{aligned}
\mathrm{hom}_{\CA}(M,N)(X,B) = \int_{A}{[M(B,A),N(X,A)]} \\
\cong [\CA^{\mathrm{op}},\CV](M(B,-),N(X,-)) \ .
\end{aligned}
\end{eqnarray}

Each $\CV$-functor $F : \CA \lra \CB$ determines modules $F_* : \CA \lra \CB$
and $F^* : \CB \lra \CA$ defined by $F_*(B,A)= \CB(B,FA)$ and $F^*(A,B)= \CB(FA,B)$.
In $\CV\text{-}\mathrm{Mod}$, we have an adjunction $F_* \dashv F^*$; that is, the $F_*$
are some of the maps in $\CV\text{-}\mathrm{Mod}$.    

We identify the 2-category $\CV\text{-}\mathrm{Cat}$ of $\CV$-categories, $\CV$-functors
and $\CV$-natural transformations with the sub-bicategory of $\CV\text{-}\mathrm{Mod}$
by restricting to modules of the form $F_* \dashv F^*$.

So in Section~\ref{Elb}, if we take $\CM = \CV\text{-}\mathrm{Mod}$ then
$\CV\text{-}\mathrm{Cat}$ provides us with a suitable sub-bicategory $\CK$.  

In this context, a triangle   
\begin{equation}\label{maintriangleVcat}
\begin{aligned}
\xymatrix{
\CA \ar[rd]_{F}^(0.5){\phantom{a}}="1" \ar[rr]^{J}  && \CB \ar[ld]^{R}_(0.5){\phantom{a}}="2" \ar@{<=}"1";"2"^-{\rho}
\\ & \CX 
}
\end{aligned}
\end{equation}
in $\CV\text{-}\mathrm{Cat}$ exhibits $R$ as a pointwise right extension of $F$ along $J$
in the sense of Section~\ref{Ple} if and only if it does in the usual sense for enriched category
theory (see \cite{DubucThesis} or \cite{KellyBook}). That is, 
\begin{equation}
R_*\cong \mathrm{hom}_{\CA}(J_*,F_*) \ ,
\end{equation}
or
\begin{equation}
\CX(X,RB) \cong [\CA^{\mathrm{op}},\CV](\CB(B,J-),\CX(X,F-)) \ . 
\end{equation}

The collage
\begin{equation}\label{collmod}
\begin{aligned}
\xymatrix{
\CA \ar[rd]_{i_1}^(0.5){\phantom{a}}="1" \ar[rr]^{M}  && \CB \ar[ld]^{i_0}_(0.5){\phantom{a}}="2" \ar@{<=}"1";"2"^-{\nu}
\\
& \mathrm{C}(M) }
\end{aligned}
\end{equation}
of a module $M$ is constructed as follows (see \cite{FiB}).
The $\CV$-category $\mathrm{C}(M)$ contains $\CA$ and $\CB$
as disjoint full sub-$\CV$-categories with inclusion $\CV$-functors
$i_1$ and $i_0$. There are no other objects. We have
$$\mathrm{C}(M)(i_0B,i_1A) = M(B,A)$$
and
$$\mathrm{C}(M)(i_1A,i_0B) = 0 \ ,$$
where $0$ is initial in $\CV$. 
Composition is obtained from composition in $\CA$ and $\CB$, and
from their actions on $M$. This collage is $\CV\text{-}\mathrm{Cat}$-compatible;
see \cite{CCEC} for the lax limit property coming from local cocompleteness
of $\CV\text{-}\mathrm{Mod}$.  

Given a span $\CA \stackrel{Q}\lla \CS \stackrel{P}\lra \CU$ in
$\CV\text{-}\mathrm{Cat}$, it follows easily that its cocomma square is
 \begin{equation}\label{cocomma}
\begin{aligned}
\xymatrix{
\CS \ar[d]_{Q}^(0.5){\phantom{aaaaa}}="1" \ar[rr]^{P}  && \CU \ar[d]^{i_0}_(0.5){\phantom{aaaaa}}="2" \ar@{<=}"1";"2"^-{\lambda}
\\
\CA \ar[rr]_-{i_1} && \mathrm{C}(P_*\circ Q^*) 
}
\end{aligned}
\end{equation}
where $\lambda$ is the mate of the $\nu$ in the collage.

As a simple application of the coend form of the enriched Yoneda Lemma (see \cite{KellyBook}), we have the following.

\begin{proposition}
Cocomma squares \eqref{cocomma} are BC.
\end{proposition}
\begin{proof}
Put $\CB = \mathrm{C}(P_*\circ Q^*)$.
By Yoneda we have $$i_0^*\circ i_{1*}(U,A)= \int^{B}{\CB(B,i_1A)\otimes \CB(i_0U,B)} \cong \CB(i_0U,i_1A)=P_*\circ Q^*(U,A) \ ,$$
as required.
\end{proof}

For ordinary categories, we have a dual which explains the definition of pointwise right extension used in \cite{FYL}. 

\begin{proposition}
For $\CV = \mathrm{Set}$, comma squares are BC.
\end{proposition}
\begin{proof}
Consider the comma square
\begin{equation}\label{comma}
\begin{aligned}
\xymatrix{
B/J \ar[d]_{D_1}^(0.5){\phantom{aaa}}="1" \ar[rr]^{D_0}  && \CU \ar[d]^{B}_(0.5){\phantom{aaa}}="2" \ar@{<=}"1";"2"^-{\lambda}
\\
\CA \ar[rr]_-{J} && \CB \ . 
}
\end{aligned}
\end{equation}
The mate of $\lambda$ is the 2-morphism $D_{0 *}\circ D_1^* \lra B^*\circ J_*$
with component 
$$\int^{(V,\beta : BV \ra JX,X)}{\CA(X,A)\times \CU(U,V)} \lra \CB(BU,JA)$$
at $(U,A)$ the function whose composite with the injection $\mathrm{in}_{(V,\beta ,X)}$ at the object 
$(V,\beta : BV \ra JX,X)$ of $B/J$ is the function
$$\CA(X,A)\times \CU(U,V) \lra \CB(BU,JA)$$
taking $(\alpha : X\ra A, \upsilon : U\ra V)$ to the composite
$$BU\stackrel{B\upsilon}\lra BV \stackrel{\beta}\lra JX \stackrel{J\alpha}\lra JA \ .$$ 
By a familiar argument with tensor products, the function
$$\CB(BU,JA) \lra \int^{(V,\beta : BV \ra JX,X)}{\CA(X,A)\times \CU(U,V) \ ,}$$
taking $\beta : BU \ra JA$ to $\mathrm{in}_{(U,\beta ,A)}(1_U,1_A)$,
gives an inverse to the mate of $\lambda$.  
\end{proof}

\section{Sketches}

Sketches on a category were introduced by Charles Ehresmann \cite{Esquisses},
extending Grothendieck topologies \cite{Art1962} and Lawvere theories \cite{Law1963}.
We generalize to sketches on an object in our bicategorical setting, 
improving on \cite{Cort}.

\begin{Definition} A {\em sketch} $\mathbb{T}$ on an object $C$ is a triangle

\begin{equation}\label{sketch}
\begin{aligned}
\xymatrix{
A \ar[rd]_{v}^(0.5){\phantom{a}}="1" \ar[rr]^{u}  && B \ar[ld]^{w}_(0.5){\phantom{a}}="2" \ar@{<=}"1";"2"^-{\tau}
\\
& C 
}
\end{aligned}
\end{equation}
in $\CK$ where $u$ is a fully faithful coopfibration. 

A {\em model} of $\mathbb{T}$ in $X$ is a 1-morphism $f : C \lra X$ in $\CK$ such that $f\tau$
exhibits $fw$ as a pointwise right extension of $fv$ along $u$. 
In particular, $f\tau$ is invertible. 
Write $\mathrm{Mdl}(\mathbb{T},X)$ for the full subcategory of $\CK(C,X)$ consisting of
these models.
\end{Definition}

\begin{proposition}
If $X$ is $u$-complete then $\mathrm{Mdl}(\mathbb{T},X)$ is the inverter of the natural transformation
$$\overline{\CK(\tau,1)}:\CK(w,1)  \Lra \mathrm{ran}_u \CK(v,1) : \CK(C,X)\lra \CK(B,X)$$
which is the mate of $\CK(u,1)\CK(w,1)\cong \CK(wu,1)  \stackrel{\CK(\tau,1)} \Lra \CK(v,1)$. 
\end{proposition}

Now suppose $\CM$ is a monoidal bicategory and that $\CK$ is closed under the monoidal structure:
that is, if $f:A\lra B$ and $g:C\lra D$ are 1-morphisms of $\CK$ then 
$f\otimes g:A\otimes C\lra B\otimes D$ is in $\CK$. 
Moreover, we suppose $\CK$ is closed with internal hom $[A,B]$. 

We say that $X$ is {\em pointwise $u$-complete} when it is $(1_K\otimes u)$-complete for every object $K$.

Notice that, for 1-morphisms $u$ and $h$ in $\CK$, we have a BC square:
\begin{equation}
\xymatrix{
H\otimes A \ar[d]_{h\otimes 1}^(0.5){\phantom{aaaaa}}="1" \ar[rr]^{1\otimes u}  && H\otimes B \ar[d]^{h\otimes 1}_(0.5){\phantom{aaaaa}}="2" \ar@{=>}"1";"2"^-{\cong}
\\
K\otimes A \ar[rr]_-{1\otimes u} && K\otimes B \ . 
}
\end{equation}
It follows that, if $X$ is pointwise $u$-complete, then 
$$\mathrm{ran}_u : \CK(K\otimes A, X)\lra \CK(K\otimes B, X)$$
is pseudonatural in $K$ and so transports to a pseudonatural transformation
 $$\CK(K,[A, X])\lra \CK(K,[B, X]) \ .$$
 By the bicategorical Yoneda lemma \cite{FiB}, this last transformation is induced
 by a 1-morphism
 $$\mathrm{r}_u : [A,X] \lra [B,X]$$
 which is right adjoint to
 $$[u,1] : [B,X] \lra [A,X] $$
 in $\CK$. 
 
 \begin{proposition}
If $X$ is pointwise $u$-complete then the inverter $\mathrm{mdl}(\mathbb{T},X)$ of the 2-morphism
$$\overline{[\tau,1]}:[w,1]  \Lra \mathrm{r}_u [v,1] : [C,X]\lra [B,X] \ ,$$
which is the mate of $[u,1][w,1]\cong [wu,1]  \stackrel{[\tau,1]} \Lra [v,1] $,
satisfies the pseudonatural equivalence
$$\mathrm{Mdl}(\mathbb{T},[K,X]) \simeq \CK(K,\mathrm{mdl}(\mathbb{T},X)) \ .$$ 
\end{proposition}

\begin{center}
--------------------------------------------------------
\end{center}

\appendix

\end{document}